\documentclass{amsart}
\usepackage[all]{xy}
\usepackage{amssymb,latexsym}
\usepackage{mltex}
\usepackage{amsmath}
\usepackage[colorlinks=true,linkcolor=blue,citecolor=blue]{hyperref}
 \usepackage{lscape}
\usepackage{enumerate}
\usepackage{amsthm}
\usepackage{hyperref}
\usepackage{multicol}

\newtheorem{thm}{Theorem}

\newtheorem{lem}{Lemma}[section]
\newtheorem{prop}[lem]{Proposition}
\newtheorem{rem}{Remark}

\newcommand{\R}{\mathbb{R}}

\newcommand{\N}{\mathbb{N}}

\begin{document}
\title{Entire downward solitons to the scalar curvature flow in Minkowski space}
\author{Pierre Bayard}
\email{bayard@ciencias.unam.mx}
\address{Facultad de Ciencias, Universidad Nacional Aut\'onoma de M\'exico
\\Av. Universidad 3000, Circuito Exterior S/N
\\Delegaci\'on Coyoac\'an, C.P. 04510, Ciudad Universitaria, CDMX, M\'exico}
\maketitle

\begin{abstract}
Existence and uniqueness in Minkowski space of entire downward translating solitons with prescribed values at infinity for a scalar curvature flow equation. The radial case translates into an ordinary differential equation and the general case into a fully non-linear elliptic PDE on $\R^n.$
\\
\\R\'ESUM\'E. Existence et unicit\'e dans l'espace de Minkowski des solitons entiers de valeurs prescrites \`a l'infini pour une \'equation de flot de courbure scalaire. Le cas radial se traduit en une \'equation diff\'erentielle ordinaire et le cas g\'eneral en une EDP elliptique totalement non-lin\'eaire sur $\R^n.$
\end{abstract}

%%%%%%%%%%%%%%%%%%%%%%%%%%%%%%%%%%%%%%%%%%%%%%%%%%%%%%%%%%%%%%%%%%%%%%%%%%%%%%%%%%%
\noindent
{\it Keywords: entire soliton, Minkowski space, scalar curvature}.\\\\
\noindent
{\it 2020 Mathematics Subject Classification:} 35J60, 53C50, 53E10.

\date{}
%%%%%%%%%%%%%%%%%%%%%%%%%%%%%%%%%%%%%%%%%%%%%%%%%%%%%%%%%%%%%%%%%%%%%%%%%%%%%%%%%%%
\maketitle\pagenumbering{arabic}
%%%%%%%%%%%%%%%%%%%%%%%%%%%%%%%%%%%%%%%%%%%%%%%%%%%%%%%%%%%%%%%%%%%%%%%%%%%%%%%%%%%
%\tableofcontents
\section{Introduction}
The Minkowski space $\R^{n,1}$ is $\R^{n+1}$ with the metric
$$dx_1^2+\cdots +dx_n^2-dx_{n+1}^2.$$
A smooth function $u:\R^n\rightarrow\R$ is spacelike if $|Du(x)|<1$ for all $x\in\R^n.$ This alternatively means that its graph $M$ is a spacelike hypersurface of $\R^{n,1}.$ In the natural chart $(x_1,\ldots,x_n),$ the metric of $M$ is
$$g_{ij}=\delta_{ij}-\partial_i u\ \partial_ju$$
and the second fundamental form is
$$h_{ij}=\frac{1}{\sqrt{1-|Du|^2}}\ \partial^2_{ij}u.$$
Since the inverse of the metric is
$$g^{ij}=\delta_{ij}+\frac{\partial_iu\ \partial_j u}{1-|Du|^2}$$
the shape operator of $M$ is given by 
\begin{equation}\label{expr shape op}
h^i_j=\frac{1}{\sqrt{1-|Du|^2}}\sum_{k=1}^n\left(\delta_{ik}+\frac{\partial_i u\ \partial_k u}{1-|Du|^2}\right)\partial^2_{kj}u.
\end{equation}
Let us denote by $\mathcal{H}_k[u]$ the (normalized) $k^{th}$ elementary symmetric function of the principal curvatures $\lambda_1,\ldots,\lambda_n$ of the graph of $u,$
$$\mathcal{H}_k[u]=\frac{k!(n-k)!}{n!}\ \sigma_k(\lambda_1,\ldots,\lambda_n).$$ 
We are interested in the scalar curvature $S[u]$ which is linked to $\mathcal{H}_2[u]$ by
$$S[u]=-n(n-1)\ \mathcal{H}_2[u],$$
and more specifically in the entire scalar curvature flow
\begin{equation}\label{entire flow}
-\frac{\dot u}{\sqrt{1-|Du|^2}}+\mathcal{H}_2[u]^{\frac{1}{2}}=H
\end{equation}
where $H>0$ is a given function. This equation expresses the evolution of a spacelike hypersurface with normal velocity given by the square root of its scalar curvature (up to a negative multiplicative constant), with a forcing term. In the paper we will suppose $n>2$ since the case $n=2$ is very different from the PDE point of view: it corresponds to the evolution of a surface by its Gauss curvature, which translates into an equation of Monge-Amp\`ere type. It is well known that Equation (\ref{entire flow}) is parabolic on the set of \emph{admissible functions}
$$\{u:\R^n\rightarrow\R\mbox{ of class }C^2\mbox{ such that }|Du|<1,\ \mathcal{H}_1[u]>0,\ \mathcal{H}_2[u]>0\mbox{ on }\R^n\},$$ 
and that the Mac-Laurin inequality
\begin{equation}\label{mac laurin}
\mathcal{H}_1[u]\geq \mathcal{H}_2[u]^{\frac{1}{2}}
\end{equation}
holds on that set. \emph{Translating solitons} are solutions of (\ref{entire flow}) moving by vertical translations, i.e. such that $\dot u=a$ for some constant $a\in\R:$ they are solutions of the so-called \emph{soliton equation}
\begin{equation}\label{soliton equation a H}
\mathcal{H}_2[u]^{\frac{1}{2}}=H+\frac{a}{\sqrt{1-|Du|^2}}.
\end{equation}
Conversely, a solution $u:\R^n\rightarrow\R$ of (\ref{soliton equation a H}) gives a solution of the flow (\ref{entire flow}) by setting 
$$u(x,t):=u(x)+at.$$
Entire translating solitons thus readily furnish natural examples of solutions for the entire parabolic problem (\ref{entire flow}). It is moreover plausible that an entire solution of (\ref{entire flow}) necessarily converges to a translating soliton if the initial hypersurface has bounded curvature and is such that $\sup_{\R^n}|Du|<1:$ this was proved for the mean curvature flow in \cite{Aa}, if $H$ is a constant. So the study of the soliton equation (\ref{soliton equation a H}) is certainly important for the study of the entire flow (\ref{entire flow}).  We are especially interested here in the existence and uniqueness of entire solutions of (\ref{soliton equation a H}), and in their asymptotic properties. Since we are seeking for admissible solutions (in order to apply elliptic methods), we need to suppose
$$H+a>0.$$ 
Let us note that for a solution $u$ of (\ref{soliton equation a H}) and $\lambda>0,$ the function $u_{\lambda}(x):=\lambda\ u(x/\lambda)$ satisfies
$$\mathcal{H}_2[u_{\lambda}]^{\frac{1}{2}}=H_\lambda+\frac{a/\lambda}{\sqrt{1-|Du_{\lambda}|^2}}$$
where $H_{\lambda}(x)=1/\lambda\ H(x/\lambda).$ The study of (\ref{soliton equation a H}) thus reduces to the three cases $a=1,$ -1, or 0. The case $a=0$ corresponds to the prescribed scalar curvature equation and was studied in \cite{Bay1}-\cite{Bay4} and \cite{U}. We focus here on the case $a=-1$ which is \emph{the downward soliton equation}
\begin{equation}\label{back soliton equation}
\mathcal{H}_2[u]^{\frac{1}{2}}=H-\frac{1}{\sqrt{1-|Du|^2}}
\end{equation}
with $H>1.$ We first prove that if $H$ is a radial function the downward soliton equation admits an entire radial solution, unique up to the addition of a constant; we will moreover specify its asymptotic behavior: 
\begin{thm}\label{thm radial}
Let $C>1$ be a constant and let $H:[0,+\infty)\rightarrow\R$ be a non-decreasing function such that $x\in\R^n\mapsto H(|x|)\in\R$ is smooth,
\begin{equation}
\inf_{[0,+\infty)}H\ >1\hspace{.5cm}\mbox{and}\hspace{.5cm}H(r)\rightarrow_{r\rightarrow+\infty} C.
\end{equation}
Then there exists a radial admissible function $u:\R^n\rightarrow\R$ solution of (\ref{back soliton equation}) in $\R^n.$ It is strictly convex and such that
$$|Du(x)|\rightarrow_{|x|\rightarrow+\infty}\sqrt{1-\frac{1}{C^2}}.$$
Assuming moreover that 
\begin{equation}\label{asympt H}
H(r)=_{r\rightarrow+\infty}C\left(1-O\left(\frac{1}{r^2}\right)\right)
\end{equation}
it satisfies
\begin{equation}\label{asymptot u}
u(x)= \sqrt{1-1/C^2}\ |x|-\frac{1}{C^2}\sqrt{\frac{n-2}{n}}\log|x|+c_0+o(1)
\end{equation}
for some constant $c_0\in\R.$ In particular, the radial admissible solution is unique up to the addition of a constant. 
\end{thm}
We then consider the Dirichlet problem for the downward soliton equation, with constant prescribed curvature and boundary values, in uniformly convex domains:
\begin{thm}\label{thm dp}
Let $\Omega\subset\R^n$ be a bounded open set with smooth boundary, uniformly convex, and let $C>1$ and $u_0\in\R$ be constant. The Dirichlet problem
\begin{equation}\label{dp 0b}\left\{\begin{array}{rcl}
{\mathcal{H}_2[u]}^{\frac{1}{2}}&=&\displaystyle{C-\frac{1}{\sqrt{1-|Du|^2}}}\hspace{.3cm}\mbox{  in  }\Omega\\
\\u&=&u_0\hspace{.3cm}\mbox{  on  }\partial\Omega
\end{array}\right.
\end{equation}
has a unique admissible solution.
\end{thm}
We finally prove that the downward soliton equation with constant curvature $C>1$ admits entire admissible solutions, with prescribed asymptotic values at infinity. Let us set
$$\widetilde{C}=\sqrt{1-1/C^2},$$
consider the sphere 
$$S_{\widetilde{C}}=\{x\in\R^n:\ |x|=\widetilde{C}\}$$ 
and fix $f:S_{\widetilde{C}}\rightarrow\R$ of class $C^2.$
\begin{thm}\label{main thm}
There exists a unique admissible solution $u:\R^n\rightarrow\R$ of 
\begin{equation}\label{soliton equation C}
\mathcal{H}_2[u]^{\frac{1}{2}}=C-\frac{1}{\sqrt{1-|Du|^2}}
\end{equation}
such that
\begin{equation}\label{asymptotic u f}
u(x)=_{|x|\rightarrow+\infty}\widetilde{C}|x|-\frac{1}{C^2}\sqrt{\frac{n-2}{n}}\log|x|+f\left(\widetilde{C}\frac{x}{|x|}\right)+o(1).
\end{equation}
\end{thm}
Similar results concerning the entire soliton equation for the mean curvature flow in Minkowski space were proved in \cite{Ji,JJLS,JJL}, \cite{Di} and \cite{SX}.
\begin{rem}
1. It may be proved that the second fundamental form of the graph of a radial solution in Theorem \ref{thm radial} decays as $1/r$ at infinity, and more specifically that
$$\mathcal{H}_1[u]\sim_{+\infty} \alpha_n\ \sqrt{C^2-1}\ \frac{1}{r}\hspace{.5cm}\mbox{and}\hspace{.5cm}\mathcal{H}_2[u]\sim_{+\infty}\beta_n\ (C^2-1)\ \frac{1}{r^2}$$
for some constants $\alpha_n,\beta_n>0$ depending only on the dimension $n$. In particular a radial soliton satisfies the pinching condition
\begin{equation}\label{rem pinching condition}
\alpha\ \mathcal{H}_1[u]\leq \mathcal{H}_2[u]^{\frac{1}{2}}\leq \beta\ \mathcal{H}_1[u]
\end{equation}
for some constants $\alpha,\beta>0.$ Such a pinching on the curvatures of the initial data was used in \cite{AS,Ch} for the study of the euclidean scalar curvature flow. We do not know if (\ref{rem pinching condition}) holds for the more general solitons obtained in Theorem \ref{main thm}.
\\2. In Theorem \ref{thm dp}, if $\Omega$ is a ball then the solution $u$ of the Dirichlet problem is the restriction of a radial soliton described in Theorem \ref{thm radial}.
\\3. By the standard elliptic regularity theory the solitons obtained in the three theorems are automatically smooth (since they are classical solutions of elliptic equations with smooth coefficients).
\\4. The uniqueness properties in the three theorems above are consequences of the standard maximum principle. For instance, in Theorem \ref{main thm}, if $u_1$ and $u_2$ are admissible solutions of (\ref{soliton equation C}) such that (\ref{asymptotic u f}) holds, and assuming by contradiction that $u_1(x_0)<u_2(x_0)$ at some point $x_0\in\R^n,$ we consider $u_1^\varepsilon:=u_1+\varepsilon$ with $\varepsilon>0$ such that 
$$u_1^\varepsilon(x_0)=u_1(x_0)+\varepsilon<u_2(x_0).$$
Since we have by (\ref{asymptotic u f})
$$(u_1^\varepsilon-u_2)(x)\rightarrow_{|x|\rightarrow +\infty}\varepsilon>0,$$
the set $\{u_1^\varepsilon\leq u_2\}$ is compact, and the maximum principle implies that $u_1^\varepsilon\equiv u_2$ on this set, a contradiction.
\\5. In Theorem \ref{main thm}, if the function $f$ is a constant then $u$ is a radial function: (\ref{soliton equation C})-(\ref{asymptotic u f}) are indeed invariant by rotations of center $0$ if $f$ is constant, and we have seen above that an admissible solution is unique (up to the addition of a constant). 
\\6. More is known on the structure of the entire solitons to the mean curvature flow in $\R^{n,1}:$ for example, they are convex and satisfy a splitting theorem, Theorem 1.2 in \cite{SX}. The question of whether these results extend to the scalar curvature flow is open.
\end{rem}
Let us quote some related papers. Entire mean curvature flow in Minkowski space was studied in \cite{Ec1,Ec2,Ec3,Aa}, and the corresponding soliton equation in \cite{Di,Ji,JJLS,JJL,SX}. The scalar curvature flow in Lorentzian geometry was first studied in \cite{Ge1,Ge2,En}. We studied in \cite{Bay4} the entire scalar curvature flow in Minkowski space, for an initial hypersurface at a bounded distance of a lightcone.

The paper is organized as follows.  In Section \ref{section radial soliton} we study the radial soliton equation and prove Theorem \ref{thm radial}, in Section \ref{section dirichlet problem} we solve the Dirichlet problem (Theorem \ref{thm dp}) and in Section \ref{section entire solutions} we study the entire solitons with prescribed values at infinity proving Theorem \ref{main thm}. A short appendix ends the paper on the asymptotic behavior of the solutions of some classical ordinary differential equations.
\section{Entire radial solitons}\label{section radial soliton}
We prove here Theorem \ref{thm radial} concerning radial solutions of the entire soliton equation: for radial solutions, the equation reduces to an ordinary differential equation. Let us note that a radial function $u:\R^n\rightarrow\R$ of class $C^2$ may be written in the form
\begin{equation}\label{u function y}
u(x)=u(0)+\int_0^{|x|}y(r)dr
\end{equation}
for some function $y:[0,+\infty)\rightarrow\R$ of class $C^1$ such that $y(0)=0.$ With this notation we have at $x\in\R^n\backslash\{0\}$
$$\partial_i u=y\ \frac{x_i}{|x|}\hspace{.5cm}\mbox{and}\hspace{.5cm}\partial^2_{ij}u=\frac{y}{|x|}\left(\delta_{ij}-\frac{x_ix_j}{|x|^2}\right)+y'\frac{x_ix_j}{|x|^2},$$
which yields, at $x=(r,0,\ldots,0),$ $r>0,$
$$Du=(y,0,\ldots,0),\hspace{1cm}D^2u=diag\left(y',\frac{y}{r},\ldots,\frac{y}{r}\right)$$
and by (\ref{expr shape op}) the following expression for the matrix of the shape operator of $M=graph(u)$
\begin{equation}\label{expr S y}
\mathcal{S}:=\frac{1}{\sqrt{1-y^2}}\ diag\left(\frac{y'}{1-y^2},\frac{y}{r},\ldots,\frac{y}{r}\right).
\end{equation}
Equation (\ref{back soliton equation}) is thus equivalent to the ordinary differential equation 
\begin{equation}\label{eqn y}
\frac{2}{n(1-y^2)}\ \frac{y}{r}\left\{\frac{y'}{1-y^2}+\frac{n-2}{2}\frac{y}{r}\right\}=\left(H(r)-\frac{1}{\sqrt{1-y^2}}\right)^2
\end{equation}
on $(0,+\infty)$ with
$$\frac{1}{\sqrt{1-y^2}}<H(r).$$
The following proposition together with (\ref{u function y}) will thus imply the first part of the theorem:
\begin{prop}\label{prop ineq y}
Under the hypotheses of Theorem \ref{thm radial}, there exists $y:[0,+\infty)\rightarrow\R$ of class $C^1$ solution of (\ref{eqn y}) on $(0,+\infty)$ and such that
\begin{equation}\label{ineq y}
y(0)=0\hspace{.5cm}\mbox{and}\hspace{.5cm}0\leq y<\sqrt{1-\frac{1}{H(r)^2}}. 
\end{equation}
It satisfies
\begin{equation}\label{lim y infty}
\lim_{r\rightarrow+\infty}y(r)=\sqrt{1-\frac{1}{C^2}}.
\end{equation}
Moreover $y'(0)=H(0)-1,$ $y'>0$ on $[0,+\infty)$  and the matrix $\mathcal{S}$ defined in (\ref{expr S y}) is positive and converges to
$$\mathcal{S}(0):=(H(0)-1)\ diag(1,1,\ldots,1)$$
as $r$ tends to $0.$ 
\end{prop}
Before the proof of the proposition, let us observe that, setting $v=y^2,$ (\ref{eqn y}) transforms to
\begin{equation}\label{eqn v}
v'=-(n-2)\frac{v(1-v)}{r}+n\ r(1-v)(H(r)\sqrt{1-v}-1)^2,
\end{equation}
and, following ideas in \cite{JJL}, we consider, for $\varepsilon>0,$ the auxiliary equation
\begin{equation}\label{eqn v epsilon}
v'=-(n-2)\frac{v(1-v)}{r+\varepsilon}+n(r+\varepsilon)(1-v)(H(r)\sqrt{1-v}-1)^2.
\end{equation}
We first solve (\ref{eqn v epsilon}) (Lemma \ref{lem v eps}), then (\ref{eqn v}) (Lemma \ref{lem v}), and afterwards prove Proposition \ref{prop ineq y}.
\begin{lem}\label{lem v eps}
(\ref{eqn v epsilon}) admits a smooth solution $v_{\varepsilon}:[0,+\infty)\rightarrow \R$ such that $v_{\varepsilon}(0)=0.$  It is non-decreasing on $[0,+\infty),$ such that 
\begin{equation}\label{ineq veps}
0<v_{\varepsilon}(r)<1-\frac{1}{H(r)^2}\hspace{.5cm}\mbox{on}\hspace{.5cm}(0,+\infty),
\end{equation}
and satisfies
$$\lim_{r\rightarrow +\infty}v_{\varepsilon}(r)=1-\frac{1}{C^2}.$$
\end{lem}
\begin{proof}
Equation (\ref{eqn v epsilon}) is of the form $v'=f_{\varepsilon}(r,v)$ with the smooth function $f_{\varepsilon}:\ (-\varepsilon,+\infty)\times (-\infty,1)\rightarrow \R$ defined by
\begin{equation}\label{expr feps}
f_{\varepsilon}(r,v)=-(n-2)\frac{v(1-v)}{r+\varepsilon}+n(r+\varepsilon)(1-v)\left(H(r)\sqrt{1-v}-1\right)^2
\end{equation}
where the function $H$ is extended smoothly to $(-\varepsilon,+\infty).$ The Cauchy-Lipschitz theorem yields a smooth solution $v_{\varepsilon}:[0,T^*)\rightarrow (-\infty,1)$ such that $v_{\varepsilon}(0)=0.$ We moreover choose $T^*$ such that $[0,T^*)$ is maximal. Let us note that the differential equation implies that
$$v_{\varepsilon}'(0)=n\varepsilon(H(0)-1)^2>0,$$
so that $v_{\varepsilon}>0$ near $r=0.$ Let us prove by contradiction that $v_{\varepsilon}>0$ on $(0,T^*):$ if $r_0>0$ is the first value such that $v_{\varepsilon}(r_0)=0$ we have $v'_{\varepsilon}(r_0)\leq 0$ and by the differential equation again
$$v'_{\varepsilon}(r_0)=n(r_0+\varepsilon)(H(r_0)-1)^2>0,$$
a contradiction. We show similarly that the second inequality in (\ref{ineq veps}) holds on $[0,T^*):$ by contradiction, if $r_0>0$ is the first point such that 
$$v_{\varepsilon}(r_0)=1-\frac{1}{H(r_0)^2}$$
the function
$$r\mapsto v_{\varepsilon}(r)-\left(1-\frac{1}{H(r)^2}\right)$$
is negative on $[0,r_0)$ and vanishes at $r_0;$ so its derivative is non-negative at $r_0,$ which implies
$$v_{\varepsilon}'(r_0)\geq 2\frac{H'(r_0)}{H(r_0)^3}\geq 0.$$
But the differential equation implies
$$v_{\varepsilon}'(r_0)=-\frac{n-2}{r_0+\varepsilon}\left(1-\frac{1}{H(r_0)^2}\right)\frac{1}{H(r_0)^2}<0,$$
again a contradiction. So (\ref{ineq veps}) holds on $(0,T^*),$ and since $H\leq C$ we deduce that $v_{\varepsilon}\in (0,1-\frac{1}{C^2})$ on $(0,T^*),$ and thus $T^*=+\infty$ (since a maximal solution leaves any compact subset of $(-\varepsilon,+\infty)\times (-\infty,1)$). We now prove that $v_{\varepsilon}'>0$ on $[0,+\infty):$ by contradiction, since $v_{\varepsilon}'(0)>0$ there would be a first point $r_1>0$ such that $v_{\varepsilon}'(r_1)=0,$ and differentiating (\ref{eqn v epsilon}) we would obtain at $r_1$
\begin{eqnarray*}
v_{\varepsilon}''(r_1)&=&(n-2)\frac{v_{\varepsilon}(1-v_{\varepsilon})}{(r_1+\varepsilon)^2}+n(1-v_{\varepsilon})\left(H(r_1)\sqrt{1-v_{\varepsilon}}-1\right)^2\\
&&+2n(r_1+\varepsilon)(1-v_{\varepsilon})\left(H(r_1)\sqrt{1-v_{\varepsilon}}-1\right)H'(r_1)\sqrt{1-v_{\varepsilon}}>0,
\end{eqnarray*}
a contradiction since ${v_{\varepsilon}'}_{|[0,r_1]}$ would also reach its minimum at $r_1$. Finally, if
$$l:=\lim_{r\rightarrow+\infty} v_{\varepsilon}(r)<1-\frac{1}{C^2},$$ 
it is easy to get from (\ref{eqn v epsilon}) that $v_{\varepsilon}'(r)\rightarrow +\infty$ as $r\rightarrow+\infty.$ This is impossible since $v_\varepsilon$ is bounded.
\end{proof}
We now pass to the limit $\varepsilon\rightarrow 0$ in (\ref{eqn v epsilon}) and obtain a solution of (\ref{eqn v}):
\begin{lem}\label{lem v}
There exists a sequence $\varepsilon_i\rightarrow_{i\rightarrow +\infty}0$ such that $(v_{\varepsilon_i})_{i\in\N}$ converges smoothly on compact subsets of $(0,+\infty).$
The limit $v=\lim_{i\rightarrow+\infty}v_{\varepsilon_i}$ is a solution of (\ref{eqn v}) on $(0,+\infty);$ $v'>0$ on $(0,+\infty),$ and $v$ is such that
\begin{equation}\label{ineg v lem}
0<v(r)<1-\frac{1}{H(r)^2}\hspace{.5cm}\mbox{on}\hspace{.5cm}(0,+\infty)
\end{equation}
and satisfies
\begin{equation}\label{lim v infty}
\lim_{r\rightarrow 0}v(r)=0\hspace{.5cm}\mbox{and}\hspace{.5cm}\lim_{r\rightarrow +\infty}v(r)=1-\frac{1}{C^2}.
\end{equation}
\end{lem}
\begin{proof}
Let us prove that for all $[a,b]\subset (0,+\infty)$ and $k\in\N$ there exists $M_k\in [0,+\infty)$ such that
\begin{equation}\label{estim vk}
\sup_{r\in[a,b]}|v_{\varepsilon}^{(k)}(r)|\leq M_k
\end{equation}
for all $\varepsilon\in (0,1).$ Recalling the expression (\ref{expr feps}) of $f_{\varepsilon}:(-\varepsilon,+\infty)\times (-\infty,1) \rightarrow\R,$ we see that $f_{\varepsilon}$ and all its partial derivatives are bounded on $[a,b]\times[0,1-1/C^2],$ independently of $\varepsilon\in (0,1).$ This gives (\ref{estim vk}) by induction: for $k=0,$ $v_{\varepsilon}$ takes values in $[0,1-1/C^2]$ and is bounded; assuming that $v_{\varepsilon},\ldots,v_{\varepsilon}^{(k)}$ are bounded on $[a,b],$ independently of $\varepsilon\in (0,1),$ so is $v_{\varepsilon}^{(k+1)}$: it is indeed obtained after $k$ derivations of $v_{\varepsilon}'=f_{\varepsilon}(r,v_{\varepsilon}),$ all of whose terms are bounded. Thus the estimates (\ref{estim vk}) hold, and extracting subsequences and using a diagonal process, we may construct  a sequence $\varepsilon_i\rightarrow_{i\rightarrow +\infty}0$ such that $(v_{\varepsilon_i})_{i\in\N}$ converges smoothly on compact subsets of $(0,+\infty).$ The smooth function $v=\lim_{i\rightarrow +\infty}v_{\varepsilon_i}$ is obviously solution of (\ref{eqn v}) on $(0,+\infty)$, is non-decreasing and satisfies
\begin{equation}\label{ineg v lem faible}
0\leq v(r)\leq1-\frac{1}{H(r)^2}\hspace{.5cm}\mbox{on}\hspace{.5cm}(0,+\infty)
\end{equation}
since so do the functions $v_{\varepsilon_i}.$ Thus $v'\geq 0$ on $(0,+\infty),$ and assuming that $v'=0$ at some point $r_0>0$, differentiating (\ref{eqn v}) we would have at $r_0$
\begin{eqnarray*}
v''(r_0)&=&(n-2)\frac{v(1-v)}{r_0^2}+n(1-v)\left(H(r_0)\sqrt{1-v}-1\right)^2\\
&&+2nr_0(1-v)\left(H(r_0)\sqrt{1-v}-1\right)H'(r_0)\sqrt{1-v}>0,
\end{eqnarray*}
a contradiction since $v'$ would reach a minimum at $r_0$ which implies $v''(r_0)=0$. Thus $v'>0$ on  $(0,+\infty)$ and (\ref{ineg v lem faible}) implies (\ref{ineg v lem}) (using also (\ref{eqn v}) to show that the last inequality in (\ref{ineg v lem faible}) is in fact a strict inequality). Moreover the limit at infinity in (\ref{lim v infty}) holds since, on the contrary, the differential equation would imply that $v'\rightarrow+\infty,$ a contradiction with $v$ bounded. For the limit as $r\rightarrow 0,$ let us note that (\ref{eqn v}) with the condition $v'\geq 0$ yields 
$$nr^2\left(H(r)\sqrt{1-v}-1\right)^2\geq (n-2)v\geq 0.$$
Since the left-hand side tends to $0$ as $r$ tends to $0,$ so does the function $v.$
\end{proof}

\noindent\emph{Proof of Proposition \ref{prop ineq y}.} Let $v:(0,+\infty)\rightarrow\R$ be the solution of (\ref{eqn v}) obtained in the previous lemma. The function $y:=\sqrt{v}:\ (0,+\infty)\rightarrow\R$ is $C^1,$ solves (\ref{eqn y}) on $(0,+\infty)$ and obviously satisfies the inequalities in (\ref{ineq y}) and the limit (\ref{lim y infty}). Setting $y(0)=0$ it is moreover continuous at 0. Let us prove that $y$ is also $C^1$ at $0:$ since the first term in the left-hand side of (\ref{eqn y}) is positive we get
$$0\leq\left(\frac{y}{r}\right)^2\leq \frac{n}{n-2}\left(H(r)\sqrt{1-y^2}-1\right)^2$$
and thus
$$0\leq\frac{y}{r}\leq \sqrt{\frac{n}{n-2}}\left(H(r)-1\right).$$
The function $\displaystyle{w=\left(\frac{y}{r}\right)^2}$ is bounded and by (\ref{eqn y}) satisfies
\begin{equation}\label{eqn w}
w'+n\frac{w}{r}=\frac{h(r)}{r}
\end{equation}
with
\begin{eqnarray*}
h(r)&=&(n-2)r^2w^2+n(1-r^2w)\left(H(r)\sqrt{1-r^2w}-1\right)^2\\
&&\rightarrow_{r\rightarrow 0}\ n\left(H(0)-1\right)^2.
\end{eqnarray*}
The unique bounded solution of (\ref{eqn w}) is
$$w(r)=\left(\int_0^rh(s)s^{n-1}ds\right)\ r^{-n}.$$
This implies that
$$w\sim_0\left(\int_0^rn(H(0)-1)^2s^{n-1}ds\right)\ r^{-n}=(H(0)-1)^2.$$
We thus have $y/r\rightarrow H(0)-1$ as $r$ tends to 0, which shows that $y$ is derivable at $0$ with $y'(0)=H(0)-1.$ Moreover, (\ref{eqn y}) again shows that
$$\frac{y}{r}\ y'\frac{1}{1-y^2}+\frac{n-2}{2}\left(\frac{y}{r}\right)^2\rightarrow_{r\rightarrow 0}\frac{n}{2}(H(0)-1)^2,$$
and, since $y/r\rightarrow H(0)-1$ and $y\rightarrow 0,$ yields $y'\rightarrow H(0)-1.$ Thus $y$ is of class $C^1$ on $[0,+\infty).$ We now prove that $y'>0$ on $[0,+\infty):$ $y'(0)=H(0)-1>0$ and we saw in Lemma \ref{lem v} that $v'=2yy'>0$ on $(0,+\infty).$ Finally, the last claim is a direct consequence of what we proved before. $\hfill\square$
\\

We now study the asymptotic behavior of the solutions of (\ref{eqn y}).
\begin{prop}
Let us assume that
$$H(r)=_{r\rightarrow+\infty}C\left(1-O\left(\frac{1}{r^2}\right)\right).$$
The solution $y$ of (\ref{eqn y}) has the following asymptotic expansion at infinity
\begin{equation}\label{form asympt 2}
y(r)=\sqrt{1-\frac{1}{C^2}}-\frac{1}{C^2}\sqrt{\frac{n-2}{n}}\ \frac{1}{r}+O\left(\frac{1}{r^2}\right).
\end{equation}
The relation (\ref{u  function y}) then gives the asymptotic behavior (\ref{asymptot u}) of the solution in Theorem \ref{thm radial}, which achieves the proof of that theorem.
\end{prop}

\begin{proof}
Let us first prove that
\begin{equation}\label{form asympt 1}
y(r)=\sqrt{1-\frac{1}{C^2}}-\frac{1}{C^2}\sqrt{\frac{n-2}{n}}\ \frac{1}{r}+o\left(\frac{1}{r}\right).
\end{equation}
Let us consider the function $z$ such that
\begin{equation}\label{z function y}
y=\sqrt{1-\frac{1}{C^2}}- \frac{1}{r}z.
\end{equation}
We observe that it is a solution of an equation of the form
\begin{equation}\label{eqn z}
z'=Az^2+B
\end{equation}
with $A$ and $B$ such that
\begin{equation}\label{lim A}
\lim_{r\rightarrow \infty}A(r)=A_0=-\frac{nC^2}{2}\sqrt{1-\frac{1}{C^2}}
\end{equation}
and
\begin{equation}\label{lim B}
\lim_{r\rightarrow \infty}B(r)=B_0=\frac{n-2}{2C^2}\sqrt{1-\frac{1}{C^2}}.
\end{equation}
To see this, let us write (\ref{eqn y}) in the form
\begin{equation}\label{eqn y bis}
2\frac{ryy'}{1-y^2}=nr^2\left(H(r)\sqrt{1-y^2}-1\right)^2-(n-2)y^2.
\end{equation}

By a direct computation using (\ref{z function y}) and since $z/r\rightarrow 0$ (by (\ref{lim y infty})), the left-hand side of (\ref{eqn y bis}) may be written in the form
$$2\frac{ryy'}{1-y^2}=\alpha_1z'+\beta_1$$
with $\alpha_1\rightarrow-2C^2\sqrt{1-\frac{1}{C^2}}$ and $\beta_1\rightarrow 0.$ Similarly, since
\begin{equation}\label{dem exp formule auxiliaire}
H(r)\sqrt{1-y^2}-1=\left(C^2\sqrt{1-\frac{1}{C^2}}+o(1)\right)\frac{z}{r}+O\left(\frac{1}{r^2}\right),
\end{equation} 
the right-hand side of (\ref{eqn y bis}) is of the form
$$nr^2\left(H(r)\sqrt{1-y^2}-1\right)^2-(n-2)y^2=\alpha_2z^2+\beta_2$$
with $\alpha_2\rightarrow nC^4(1-\frac{1}{C^2})$ and $\beta_2\rightarrow -(n-2)(1-\frac{1}{C^2}).$ Equations (\ref{eqn z}), (\ref{lim A}) and (\ref{lim B}) follow. Proposition \ref{prop 2 app 1} in the appendix then implies that
$$\lim_{r\rightarrow+\infty}z(r)=\sqrt{-\frac{B_0}{A_0}}=\frac{1}{C^2}\sqrt{\frac{n-2}{n}},$$
which gives (\ref{form asympt 1}). We now prove (\ref{form asympt 2}). Let us consider the function $\alpha$ such that
$$y=\sqrt{1-\frac{1}{C^2}}-\frac{1}{C^2}\sqrt{\frac{n-2}{n}}\ \frac{1}{r}-\frac{\alpha}{r^2}.$$ 
By (\ref{form asympt 1}), $\alpha/r\rightarrow 0$ as $r$ tends to infinity.
Let us see that $\alpha$ is solution of an equation of the form
\begin{equation}\label{pf asympt y eqn alpha}
\alpha'+a\alpha=b
\end{equation}
with $a\rightarrow a_0>0$ and $b$ bounded; it will imply that $\alpha$ is bounded by Proposition \ref{prop 3 app 1} in the appendix, which will finish the proof of (\ref{form asympt 2}). To see that (\ref{pf asympt y eqn alpha}) holds, we write 
(\ref{eqn y}) in the form
\begin{equation}\label{eqn y ter}
2\frac{r^2yy'}{1-y^2}=r\left(nr^2\left(H(r)\sqrt{1-y^2}-1\right)^2-(n-2)y^2\right).
\end{equation}
The left-hand side of (\ref{eqn y ter}) may be written
$$2\frac{r^2yy'}{1-y^2}=-\frac{2y}{1-y^2}\alpha'+\frac{2y}{1-y^2}\left(r^2y'+\alpha'\right);$$
since
$$\frac{2y}{1-y^2}\rightarrow 2C^2\sqrt{1-\frac{1}{C^2}}\hspace{.5cm}\mbox{and}\hspace{.5cm}r^2y'+\alpha'=\frac{1}{C^2}\sqrt{\frac{n-2}{n}}+2\frac{\alpha}{r}\rightarrow\frac{1}{C^2}\sqrt{\frac{n-2}{n}},$$
it is of the form 
\begin{equation}\label{pf asympt y 1}
2\frac{r^2yy'}{1-y^2}=c_1\alpha'+c_2
\end{equation}
with $c_1\rightarrow -2C^2\sqrt{1-\frac{1}{C^2}}$ and $c_2$ bounded. For the right-hand side of (\ref{eqn y ter}), using the following improvement of the expansion (\ref{dem exp formule auxiliaire}) 
$$H(r)\sqrt{1-y^2}-1=\left(C^2\sqrt{1-\frac{1}{C^2}}+O\left(\frac{1}{r}\right)\right)\frac{z}{r}+O\left(\frac{1}{r^2}\right)$$
with
$$z=\frac{1}{C^2}\sqrt{\frac{n-2}{n}}+\frac{\alpha}{r}$$
bounded, it is easy to show that  
$$nr^3\left(H(r)\sqrt{1-y^2}-1\right)^2=(n-2)\left(1-\frac{1}{C^2}\right)r+c_3\alpha+O(1)$$
where $c_3\rightarrow 2\sqrt{n(n-2)}(C^2-1),$ and
$$r(n-2)y^2=(n-2)\left(1-\frac{1}{C^2}\right)r+O(1),$$
so that
\begin{equation}\label{pf asympt y 2}
r\left(nr^2\left(H(r)\sqrt{1-y^2}-1\right)^2-(n-2)y^2\right)=c_3\alpha+c_4
\end{equation}
with $c_3\rightarrow 2\sqrt{n(n-2)}(C^2-1)$ and $c_4$ bounded. Equation (\ref{eqn y ter}) with (\ref{pf asympt y 1}) and (\ref{pf asympt y 2}) implies (\ref{pf asympt y eqn alpha}), and the result.
\end{proof}

\section{The Dirichlet problem}\label{section dirichlet problem}
We consider here the Dirichlet problem with constant curvature and boundary values (\ref{dp 0b}), and prove Theorem \ref{thm dp}. We may suppose that the boundary condition is $u_0=0.$
\subsection{Method of resolution}
We use the method of continuity: we consider, for $\sigma \in[0,1],$ the family of Dirichlet problems
\begin{equation}\label{dp sigma}
\left\{\begin{array}{rcl}
{\mathcal{H}_2[u]}^{\frac{1}{2}}&=&\displaystyle{C-\frac{\sigma}{\sqrt{1-|Du|^2}}}\hspace{.3cm}\mbox{  in  }\Omega\\
\\u&=&0\hspace{.3cm}\mbox{  on  }\partial\Omega,
\end{array}\right.
\end{equation}
and set
$$S:=\{\sigma\in[0,1]:\ \exists u\in C^{\infty}(\overline{\Omega}) \mbox{ admissible solution of }(\ref{dp sigma}) \}.$$
1. $S\neq\emptyset$: by \cite{Bay1,Bay2,U} the Dirichlet problem (\ref{dp sigma}) with $\sigma=0$ admits a smooth admissible solution.
\\2. $S$ is open in $[0,1]:$ this is a consequence of the ellipticity of $\mathcal{H}_2$ at an admissible function and the linear theory.
\\3. $S$ is closed in $[0,1]$ if the following a priori estimates hold: there exist $\alpha\in (0,1)$ and controlled constants $M\in\R^+,$ $\theta\in (0,1)$ and $\delta>0$ such that
\begin{equation}\label{a priori estimates 1}
 \|u\|_{2,\alpha}\leq M,\hspace{1cm}\sup_{\overline{\Omega}}|Du|\leq 1-\theta
 \end{equation}
and
\begin{equation}\label{a priori estimates 2}
\mathcal{H}_2[u]\geq\delta
\end{equation}
for all the admissible solutions of (\ref{dp sigma}). The constants $M,$ $\theta$ and $\delta$ are assumed to be locally independent of $\sigma:$ if $\sigma_0\in(0,1]$ is given, $M,$ $\theta,$ and $\delta$ are independent of $\sigma$ in some neighborhood of $\sigma_0,$ and only depend on $\Omega,$ $C,$ $n$ and $\sigma_0.$ These estimates indeed imply that $S$ is closed in $[0,1]$: if $(\sigma_n)_{n\in\N}\in S^{\N}$ converges to $\sigma_0\in(0,1],$ and $u_n$ is an admissible solution of (\ref{dp sigma}) with $\sigma=\sigma_n,$ the estimates (\ref{a priori estimates 1}) allow to suppose that $(u_n)_{n\in\N}$ converges to a spacelike function $u$ in $C^2(\overline{\Omega}),$ which is a solution of (\ref{dp sigma}) with $\sigma=\sigma_0;$ moreover, the a priori estimate (\ref{a priori estimates 2}) implies that $u$ is admissible, since, passing to the limit in the Mac-Laurin inequality (\ref{mac laurin}) for $u_n$, we obtain
$$\mathcal{H}_1[u]\geq\mathcal{H}_2[u]^{\frac{1}{2}}\geq \delta^{\frac{1}{2}}>0.$$ 
Finally $u$ is smooth, by the standard elliptic regularity theory. This implies that $\sigma_0$ belongs to $S.$

So $S=[0,1]$ and the Dirichlet problem (\ref{dp 0b}) ((\ref{dp sigma}) with $\sigma=1$) admits a smooth admissible solution.\\

Let us note that an estimate of the form $\displaystyle{\mathcal{H}_2[u]\geq\delta}$ for an admissible solution of (\ref{dp sigma}) is equivalent to an estimate
\begin{equation}\label{estim nu precise}
\frac{1}{\sqrt{1-|Du|^2}}\leq C'/\sigma
\end{equation}
for some constant $C'$ strictly smaller than $C$ since by (\ref{dp sigma})
\begin{equation}\label{equiv estimates H2 nu}
\mathcal{H}_2[u]\geq\delta\hspace{.5cm}\mbox{ if and only if}\hspace{.5cm} \frac{1}{\sqrt{1-|Du|^2}}\leq \frac{C-\sqrt{\delta}}{\sigma}.
\end{equation}
Note also that (\ref{estim nu precise}) yields the gradient estimate required in (\ref{a priori estimates 1}). We will thus focus below on the proof of the gradient estimate (\ref{estim nu precise}) and afterwards on the estimate of the second derivatives (the $C^0$ estimate is trivial). The $C^{2,\alpha}$ estimate will then follow from the Evans, Krylov and Safonov estimates \cite{GT}.

\subsection{The gradient estimate}
We prove here the gradient estimate (\ref{estim nu precise}) in two steps: we first prove that the gradient cannot attain its maximum at a point interior to the domain, and then deal with the estimate at the boundary. 
\subsubsection{Maximum principle for the gradient}\label{subsection max pple gradient}
We first show that the maximum of 
$$\nu:=\frac{1}{\sqrt{1-|Du|^2}}$$ 
is attained at the boundary of $\Omega,$ i.e. that
\begin{equation}\label{pp max nu}
\sup_{\overline{\Omega}}\nu=\sup_{\partial\Omega}\nu.
\end{equation}
We will do computations on the hypersurface $M=graph(u),$ and use the usual index notation and summation convention for tensors. We will denote, for $F=\mathcal{H}_2^{\frac{1}{2}}$ considered as a function of the shape operator $(h^i_j)_{i,j},$
$$F_i^j:=\frac{\partial F}{\partial h^i_j}\left((h^i_j)_{i,j}\right)\hspace{.3cm}\mbox{and}\hspace{.3cm}F^{ij}:=g^{ik}F_k^j;$$
they are tensors on $M.$ Denoting by $e^0_{n+1}$ the last vector of the canonical basis of $\R^{n,1}$ and by $N$ the upward unit normal to $M,$ we have
\begin{equation}\label{nu T N}
\nu=-\langle e^0_{n+1},N\rangle
\end{equation}
and
\begin{equation}\label{e_n+1 T nu}
e^0_{n+1}=T+\nu N
\end{equation}
where $T:=-\nabla u\ \in TM$  is the opposite of the intrinsic gradient of $u$ on $M.$ We in particular have the formula
\begin{equation}\label{nu grad u}
\nu^2=1+|\nabla u|^2.
\end{equation}
Let us denote by $t^k,$ $k=1,\ldots,n$ the components of the vector $T$ in some frame. Differentiating (\ref{nu T N}) twice we get the formulas
\begin{equation}\label{nui}
\nu_i=-t^kh_{ki}
\end{equation}
and 
\begin{equation}\label{nuij}
\nu_{ij}=-\left(t^kh_{ki;j}+t^k_{;j}h_{ki}\right)
\end{equation}
where we use a semi-colon to denote the covariant derivative of tensors on $M.$ Let us also recall the Gauss formula $u_{ij}=\nu h_{ij},$ which may be written here
\begin{equation}\label{nabla T nu h}
t^k_{;j}=-\nu h^k_j.
\end{equation}
Since the tensor $h_{ij;k}$ is symmetric in all the indices (the Codazzi equations), we thus have
$$F^{ij}\nu_{ij}=-t^kF^{ij}h_{ij;k}+\nu F^{ij}h^k_jh_{ki}.$$
Since $F=C-\sigma\nu,$ we have
$$F^{ij}h_{ij;k}=-\sigma\nu_k,$$
which yields
\begin{equation}\label{Fijnuij}
F^{ij}\nu_{ij}=\sigma t^k\nu_k+\nu F^{ij}h^k_jh_{ki}.
\end{equation}
Since the second term in the right-hand side is positive, it is not possible that simultaneously $\nu_k=0$ and $F^{ij}\nu_{ij}\leq 0,$ and $\nu$ cannot achieve its maximum at a point interior to the domain $\Omega;$ the property (\ref{pp max nu}) follows.\\

\subsubsection{Boundary estimate of the gradient}
We will construct upper and lower barriers at the boundary for the Dirichlet problems (\ref{dp sigma}). The construction will rely on the following results concerning radial solutions :
\begin{lem}\label{lem def y sigma}
For all $\sigma\in [0,1],$ there exists a function $y_{\sigma}:[0,+\infty)\rightarrow\R$ of class $C^1$ solution of
\begin{equation}\label{eqn y sigma}
\frac{2}{n(1-y^2)}\ \frac{y}{r}\left\{\frac{y'}{1-y^2}+\frac{n-2}{2}\frac{y}{r}\right\}=\left(C-\frac{\sigma}{\sqrt{1-y^2}}\right)^2
\end{equation}
on $(0,+\infty)$ and such that
\begin{equation}\label{ineq y sigma}
y_{\sigma}(0)=0\hspace{.5cm}\mbox{and}\hspace{.5cm}0\leq y_{\sigma}<\sqrt{1-\frac{\sigma^2}{C^2}}.
\end{equation}
It satisfies
\begin{equation}\label{lim y sigma infty}
\lim_{r\rightarrow+\infty}y_{\sigma}(r)=\sqrt{1-\frac{\sigma^2}{C^2}}.
\end{equation}
Moreover $y_{\sigma}'>0$ on $[0,+\infty),$ and $y_{\sigma}'(0)=C-\sigma$ for $\sigma\in (0,1].$
\end{lem}
\begin{proof}
The case $\sigma=0$ corresponds to the radial prescribed scalar curvature equation and has the trivial solution (corresponding to an hyperboloid) 
$$y_0(r)=\frac{r}{\sqrt{r^2+\frac{1}{C^2}}}.$$
We thus assume that $\sigma\in (0,1]$. Let $y$ be the solution of (\ref{eqn y}) with  $H=C/\sigma$ (Proposition \ref{prop ineq y} with the constant $C/\sigma$ instead of $C$ and $H=C/\sigma$), and set, for $r\in[0,+\infty),$ $y_{\sigma}(r):=y\left(\sigma  r\right).$ It is clearly a solution of (\ref{eqn y sigma}). The other properties of $y_\sigma$ directly follow from the properties of $y.$
\end{proof}
\begin{lem}
If $\sigma,\sigma'\in [0,1]$ are such that $\sigma<\sigma',$ the functions $y_\sigma$ and $y_{\sigma'}$ defined in Lemma \ref{lem def y sigma} are such that
\begin{equation}\label{ineq y sigmas}
y_{\sigma}>y_{\sigma'}\hspace{.5cm}\mbox{on}\hspace{.3cm}(0,+\infty). 
\end{equation}
Moreover, for all $R>0$ there exists a constant $c_0=c_0(n,C,R)>0$ such that 
\begin{equation}\label{ineq y sigma R}
\sqrt{1-\frac{\sigma^2}{(C-c_0)^2}}\geq y_{\sigma}(R)\geq 0.
\end{equation}
for all $\sigma\in(0,1].$
\end{lem}
\begin{proof}
We first prove (\ref{ineq y sigmas}). Using Lemma \ref{lem def y sigma} we see that
$$y_{\sigma}(0)=y_{\sigma'}(0)\hspace{.5cm}\mbox{and}\hspace{.5cm}y_{\sigma}'(0)>y_{\sigma'}'(0),$$
so that (\ref{ineq y sigmas}) holds in a neighborhood of $0.$ By contradiction, if (\ref{ineq y sigmas}) does not hold on $(0,+\infty)$ there exists a first point $r>0$ such that 
\begin{equation}\label{ysigma ysigmap 1}
y_{\sigma}(r)=y_{\sigma'}(r).
\end{equation} 
Since $y_{\sigma}>y_{\sigma'}$ on $(0,r),$ we also have 
\begin{equation}\label{ysigma ysigmap 2}
y_{\sigma}'(r)\leq y_{\sigma'}'(r).
\end{equation}
The conditions (\ref{ysigma ysigmap 1}), (\ref{ysigma ysigmap 2}) and the differential equation (\ref{eqn y sigma}) satisfied by $y_\sigma$ and $y_{\sigma'}$ readily imply that $\sigma\geq\sigma',$ a contradiction. We now prove (\ref{ineq y sigma R}). Let us fix $R>0.$ Since $y_{\sigma}'>0$ on $(0,+\infty),$ we directly have from (\ref{eqn y sigma}) that
$$C-\frac{\sigma}{\sqrt{1-y_{\sigma}(R)^2}}\geq\sqrt{\frac{n-2}{n}}\ \frac{y_{\sigma}(R)}{R}.$$
Since $y_{\sigma}(R)\geq y_1(R)$ (by (\ref{ineq y sigmas})), (\ref{ineq y sigma R}) holds with $c_0:=\sqrt{\frac{n-2}{n}}\ \frac{y_{1}(R)}{R}.$
\end{proof}
We now estimate the gradient on the boundary of the domain:
\begin{prop}
There exists a positive controlled constant $C'<C$ such that
\begin{equation}\label{dp grad est min}
\sup_{\partial\Omega}\left|Du\right|\leq \sqrt{1-\frac{\sigma^2}{C'^2}}
\end{equation}
for all the admissible solutions of the Dirichlet problems (\ref{dp sigma}). The constant $C'$ only depends on $\Omega,$ $C$ and $n.$
\end{prop}
\begin{proof}
Since $\Omega$ is uniformly convex we may find uniform constants $R^+,R^->0$ and, for each boundary-point $x_0\in\partial\Omega,$ two balls $B(x_0^+,R^+)$ and $B(x_0^-,R^-)$ tangent to $\partial\Omega$ at $x_0$ and satisfying 
\begin{equation}\label{inclusion balls}
B(x_0^+,R^+)\ \subset\ \Omega\ \subset\ B(x_0^-,R^-);
\end{equation}
$x_0^+,$ $x_0^-\in\R^n$ stand for the centers and $R^+,R^-$ for the radii of the balls. We consider, for a fixed $x_0\in\partial\Omega,$ the radial functions 
$$v^{\pm}_{\sigma}(x)=\int_0^{|x-x_0^{\pm}|}y_{\sigma}(r)dr-\int_0^{R^{\pm}}y_{\sigma}(r)dr$$ 
of center $x_0^{\pm};$ by construction, $v^{\pm}_{\sigma}$ are solutions of the prescribed curvature equation (\ref{dp sigma}) on $\R^n$ and vanish on $\partial B(x_0^{\pm},R^{\pm}).$ In particular, they are upper and lower barriers of $u$ at $x_0:$ since $u$ vanishes on $\partial\Omega$ and by (\ref{inclusion balls}) we have $v^{+}_{\sigma}\geq u\geq v^{-}_{\sigma}$ on $\partial\Omega$ and $v^{+}_\sigma(x_0)=v^{-}_\sigma(x_0)=u(x_0),$ and thus $v^{+}_\sigma\ge u\geq v^{-}_\sigma$ on $\Omega$ by the comparison principle. Denoting by $\partial_{\vec{n}}$ the derivative with respect to the inner unit normal to $\partial\Omega$ at $x_0,$ this implies that 
$$\partial_{\vec{n}}v^{+}_{\sigma}(x_0)\geq \partial_{\vec{n}}u(x_0)\geq \partial_{\vec{n}}v^{-}_{\sigma}(x_0).$$ 
Using (\ref{ineq y sigma R}) we deduce that
$$0\geq \partial_{\vec{n}}u(x_0)\geq -y_{\sigma}(R^-)\geq-\sqrt{1-\frac{\sigma^2}{(C-c_0)^2}}$$
and (\ref{dp grad est min}) follows with $C'=C-c_0$ since $u$ is constant on the boundary. The constant $C'$ depends on $n,C$ and $R^-.$
\end{proof}
\subsubsection{Global gradient estimate}
We readily deduce from the last proposition and (\ref{pp max nu}) the global bound (\ref{estim nu precise}) of the gradient:
\begin{prop}
There exists a controlled constant $C'<C$ such that
\begin{equation}\label{dp grad est}
\nu\leq C'/\sigma
\end{equation}
for all the admissible solutions of the Dirichlet problems (\ref{dp sigma}). The constant $C'$ only depends on $\Omega,$ $C$ and $n.$
\end{prop}
Recall that this estimate also implies the required estimate (\ref{a priori estimates 2}) (see (\ref{equiv estimates H2 nu})).

\subsection{The $C^2$ estimate}
\subsubsection{Maximum principle for the second derivatives}\label{section max pp C2}
The estimate relies on the $C^2$ estimate of Urbas in \cite{U}. We assume that the estimate (\ref{a priori estimates 2}) holds (it is equivalent to (\ref{dp grad est})), which implies that the largest principal curvature $\lambda_1$ of the graph of $u$ is bounded below, since we have
$$\lambda_1\geq\mathcal{H}_1\geq \mathcal{H}_2^{\frac{1}{2}}$$
as a consequence of the admissibility of $u.$ The aim is to bound $\lambda_1$ from above. Let us write the prescribed curvature equation (\ref{dp sigma}) in the form
\begin{equation}\label{eqn not Urbas}
F=\psi
\end{equation}
with $F=\mathcal{H}_2^{\frac{1}{2}}$ and $\psi=C-\sigma\nu.$ The method in \cite{U} consists in applying the maximum principle for the test function
\begin{equation}\label{def W tilde}
\widetilde{W}(x,\xi)=\eta^{\beta}\sum_{i,j}h_{ij}\xi_i\xi_j
\end{equation}
defined for $x\in M$ and $\xi\in T_xM$ such that $|\xi|=1,$ with $\eta=\exp(\Phi)$ where $\Phi:\overline{\Omega}\rightarrow\R$ is a spacelike strictly convex function and $\beta>0$ is a constant to be chosen sufficiently large. In \cite{U} the function $\psi$ in the right-hand side of (\ref{eqn not Urbas}) does not depend on the gradient of the unknown function $u;$ here $\psi$ depends on the gradient, and we have to ensure that the third derivatives of $u$ which appear differentiating (\ref{eqn not Urbas}) twice may still be controlled; this is the aim of the following lemma: 
\begin{lem}\label{lem est W tilde}
Assume that $\widetilde{W}$ reaches its maximum at $(x_0,e_1)\in TM.$ We have the estimate
$$\frac{\psi_{11}}{\lambda_1}\geq -c(\beta+\lambda_1)$$
where $c$ is a controlled constant and $\lambda_1$ is the largest principal curvature of the graph of $u$ at $x_0.$ 
\end{lem}
Note that the corresponding estimate used in \cite{U} (for $\psi$ independent of the gradient but arbitrarily given) is $\psi_{11}/\lambda_1\geq -c(1+\lambda_1)$ (see inequality (2.8) in \cite{U}).
\begin{proof}
By definition of $\widetilde{W}$ the vector $e_1$ is a principal direction of $M$ corresponding to the largest principal curvature $\lambda_1$ at $x_0.$ Let us assume that $e_1,e_2, \ldots,e_n$ is an orthonormal basis of $T_{x_0}M$ formed by principal directions. By (\ref{nuij}) and since $h_{k1;1}=h_{11;k}$ (Codazzi equations), we have 
\begin{equation}\label{est psi11 1}
\psi_{11}=-\sigma\nu_{11}=\sigma( t^kh_{11;k}+t^k_{;1}h_{k1}).
\end{equation}
We estimate the first term in the right-hand side: the condition expressing that $\widetilde{W}$ reaches an extremum at $(x_0,e_1)$ is
$$\beta\frac{\eta_i}{\eta}+\frac{h_{11;i}}{h_{11}}=0$$
for $i=1,\ldots,n$ (see details in \cite{U}). Since $\eta_i/\eta$ is bounded (the gradient is supposed to be controlled), we get
\begin{equation}\label{est psi11 3}
|h_{11;i}|\leq c\beta\lambda_1
\end{equation}
for a controlled constant $c.$ We then note that by (\ref{nabla T nu h}) the last term of (\ref{est psi11 1}) is
\begin{equation}\label{est psi11 2}
t^k_{;1}h_{k1}=-\nu h_1^kh_{k1}=-\nu\lambda_1^2.
\end{equation}
Formula (\ref{est psi11 1}) with the estimates (\ref{est psi11 3}) and (\ref{est psi11 2}) implies the lemma ($T=-\nabla u$ and $\nu=\sqrt{1+|\nabla u|^2}$ are under control since so is the gradient).
\end{proof}
With this lemma, we may follow the lines of the $C^2$ estimate in \cite{U} and obtain the $C^2$ estimate in our case as well, without any relevant modification: the final key estimate (2.16) in \cite{U} is replaced here by
$$0\geq\frac{c_2\beta\lambda_1}{\eta}+f_1\lambda_1^2-\frac{c(\beta+\beta^2)f_1}{\eta^2}-\frac{c\beta}{\eta}-c(\beta+\lambda_1)$$
with $f_1=\partial \sigma_2^{1/2}/\partial\lambda_1,$ and we obtain a bound for $\eta^\beta\lambda_1$ by fixing $\beta\geq 1$ large such that
$$\frac{c_2\beta}{\eta}\geq 2c.$$
\begin{rem}\label{rem int C2} 
As in \cite{U} this estimate may also be carried out if $\eta$ is a function such that $\eta=0$ on $\partial\Omega$ and $\eta>0$ in $\Omega.$ We will use this fact for the interior $C^2$ estimate in Section \ref{section entire solutions}, required in the construction of entire solutions.
\end{rem}
\subsubsection{Boundary estimate of the second derivatives} 
The $C^2$ estimate at the boundary is a consequence of the $C^2$ boundary estimate in \cite{Bay2}, since the estimate in that paper is not affected if the right-hand term in the prescribed curvature equation $\mathcal{H}_m=H$ depends on the gradient of $u.$ More precisely, the prescribed curvature equation is written in \cite{Bay2} in the form
\begin{equation}\label{forme equation article dirichlet}
\sum_{I,J}A^{IJ}(Du)(D^2u)_{IJ}=f(x,Du)
\end{equation}
with
\begin{equation}\label{forme f article dirichlet}
f(x,p)=\frac{n!}{m!(n-m)!}(1-|p|^2)^{\frac{m}{2}}\ H(x)
\end{equation}
where $A^{IJ}(Du)$ and $(D^2u)_{IJ}$ are the minors of indices the $m$-tuples $I,J$ of, respectively, $(\delta_{ij}+\frac{u_iu_j}{1-|Du|^2})_{ij}$ and the hessian $D^2u.$ 
It appears that the $C^2$ boundary estimate in that paper is independent of the special form (\ref{forme f article dirichlet}) and only requires that $f$ is bounded below by a controlled positive constant and that $f$ and its first derivatives are bounded functions on the set $\overline{\Omega}\times \overline{B}(0,1-\theta)$ where $\theta\in (0,1)$ is a constant such that $\sup_{\Omega}|Du|\leq 1-\theta.$ We need here to estimate solutions of (\ref{forme equation article dirichlet}) with  
$$f(x,p)=\frac{n!}{m!(n-m)!}(1-|p|^2)^{\frac{m}{2}}\left(C-\frac{\sigma}{\sqrt{1-|p|^2}}\right)^2$$
where $\sigma$ belongs to $(0,1].$ We have by (\ref{dp grad est}) that 
$$f(x,Du)\geq  \frac{n!}{m!(n-m)!}\frac{\sigma^m}{C'^m}(C-C')^2$$
for a solution $u$ of the Dirichlet problem (\ref{dp sigma}), and the right-hand term may be bounded below independently of $\sigma$ on some neighborhood of a fixed $\sigma_0\in (0,1],$ e.g. on $(\sigma_0/2,1].$ Moreover $f$ and its derivatives are clearly bounded on $\overline{\Omega}\times\overline{B}(0,\sqrt{1-1/C^2}).$ So the $C^2$ boundary estimate in \cite{Bay2} applies here and gives the required estimate.
\subsubsection{Global $C^2$ estimate}
Gathering the results obtained above we get the following $C^2$ estimate on $\overline{\Omega}:$
\begin{prop}
Let us fix $\sigma_0\in (0,1].$ There exists a controlled constant $C''$ such that
\begin{equation}\label{global C2 est}
\sup_{\overline{\Omega}}|D^2u|\leq C''
\end{equation}
for all the admissible solutions of the Dirichlet problems (\ref{dp sigma}) with $\sigma\in (\sigma_0/2,1].$ The constant $C''$ only depends on $\Omega,C,C',n$ and $\sigma_0.$
\end{prop}
\section{Entire solitons of prescribed values at infinity: proof of Theorem \ref{main thm}}\label{section entire solutions}
\subsection{The method of resolution}
We will first construct lower and upper barriers $\underline{u},\overline{u}:\R^n\rightarrow \R$ with $\underline{u}$ strictly convex, $\underline{u}<\overline{u},$ and such that 
\begin{equation}\label{asympt under over u}
\underline{u}(x),\overline{u}(x)=\widetilde{C}|x|-\frac{1}{C^2}\sqrt{\frac{n-2}{n}}\log|x|+f\left(\widetilde{C}\frac{x}{|x|}\right)+o(1)
\end{equation}
as $|x|$ tends to infinity. They will be defined as supremum and infimum of families of radial solutions of (\ref{soliton equation C}). Let $\Phi:\R^n\rightarrow\R$ be a strictly convex smooth function such that $\underline{u}\leq\Phi<\overline{u}.$ We consider the uniformly convex domains $\Omega_m:=\Phi^{-1}(-\infty,m),$ and a sequence $(u_m)_{m\in\N}$ of smooth admissible solutions of the Dirichlet problems
\begin{equation}\label{seq d problems}
\left\{\begin{array}{rcl}
\mathcal{H}_2[u]^{\frac{1}{2}}&=&C-\frac{1}{\sqrt{1-|Du|^2}}\hspace{.5cm} \mbox{in}\ \Omega_m\\
u&=&m\hspace{.5cm} \mbox{on}\ \partial \Omega_m
\end{array}\right.
\end{equation}
whose existence (and uniqueness) is granted by Theorem \ref{thm dp}. We have 
\begin{equation}\label{um between barriers}
\underline{u}\leq u_m\leq \overline{u}
\end{equation} 
on $\Omega_m.$ We need the following interior a priori estimates of the gradient and the higher derivatives of $u_m:$ if $K\subset\R^n$ is a compact subset,
\begin{equation}\label{required a priori interior gradient estimate}
\sup_K \frac{1}{\sqrt{1-|Du_m|^2}}\leq C_K<C,
\end{equation}
and, for all $i\in\N,$
\begin{equation}\label{required a priori interior Ci estimate}
 \sup_K \sum_{|\alpha|=i}\left|\partial^{\alpha}u_m\right|\leq C(K,i)
\end{equation}
for constants $C_K$ and $C(K,i)$ independent of $m\in\N.$ With these estimates at hand, the Arzela-Ascoli theorem and a diagonal process give an entire spacelike smooth function $u:\R^n\rightarrow \R$ solution of (\ref{soliton equation C}) between the barriers, thus satisfying the asymptotic condition (\ref{asymptotic u f}).
\subsection{The construction of the barriers}
The ideas of the construction rely on the papers \cite{T,SX} concerning the mean curvature operator $\mathcal{H}_1.$ Since the proofs are identical for the scalar curvature operator $\mathcal{H}_2,$ we only recall the definitions and refer to these papers for details (such a construction was also used for $\mathcal{H}_2$ in \cite{Bay3,Bay4}): we extend $f$ to $\R^n\backslash\{0\}$ by setting $f(\widetilde{C}x)=f(\widetilde{C}x/|x|),$ let $M$ be a constant such that
$$\left|f(\widetilde{C}x)-f(\widetilde{C}y)-df_{\widetilde{C}y}(\widetilde{C}x-\widetilde{C}y)\right|\leq M\left|\widetilde{C}x-\widetilde{C}y\right|^2$$
for all $x,y\in\mathbb{S}^{n-1},$ and set
$$p_1(\widetilde{C}y):=Df(\widetilde{C}y)+2M\widetilde{C}y,\hspace{.5cm} p_2(\widetilde{C}y):=Df(\widetilde{C}y)-2M\widetilde{C}y.$$
Let us consider the radial solution $\psi:\R^n\rightarrow\R$ of (\ref{soliton equation C}) such that
$$\psi(x)= \widetilde{C}|x|-\frac{1}{C^2}\sqrt{\frac{n-2}{n}}\log|x|+o(1)$$
as $|x|$ tends to $+\infty$ (Theorem \ref{thm radial}). We set
$$\underline{z}(x,\widetilde{C}y):=f(\widetilde{C}y)-p_1(\widetilde{C}y)\cdot\widetilde{C}y+\psi(x+p_1(\widetilde{C}y))$$
and
$$\overline{z}(x,\widetilde{C}y):=f(\widetilde{C}y)-p_2(\widetilde{C}y)\cdot\widetilde{C}y+\psi(x+p_2(\widetilde{C}y)).$$
They satisfy the following properties: for all $x,y\in\mathbb{S}^{n-1},$
$$f(\widetilde{C}x)\geq \underline{z}(rx,\widetilde{C}y)-\widetilde{C}r+\frac{1}{C^2}\sqrt{\frac{n-2}{n}}\log r$$
and
$$f(\widetilde{C}x)\leq \overline{z}(rx,\widetilde{C}y)-\widetilde{C}r+\frac{1}{C^2}\sqrt{\frac{n-2}{n}}\log r$$
as $r\rightarrow +\infty$ (see \cite{SX}). Setting, for $x\in\R^n,$
$$\underline{u}(x):=\sup_{y\in\mathbb{S}^{n-1}} \underline{z}(x,\widetilde{C}y)\hspace{.5cm}\mbox{and}\hspace{.5cm}\overline{u}(x):=\inf_{y\in\mathbb{S}^{n-1}} \overline{z}(x,\widetilde{C}y)$$
we obtain barriers $\underline{u},\overline{u}:\R^n\rightarrow\R,$ with $\underline{u}$ strictly convex, so that $\underline{u}\leq\overline{u}$ (by the maximum principle) and (\ref{asympt under over u}) holds. Note finally that we may assume that $\underline{u}<\overline{u}$ on $\R^n,$ since otherwise it would exist $x\in\R^n$ and $y_1,y_2\in\mathbb{S}^{n-1}$ so that $\underline{z}(x,\widetilde{C}y_1)= \overline{z}(x,\widetilde{C}y_2);$ we would then have $\underline{z}(.,\widetilde{C}y_1)\equiv \overline{z}(.,\widetilde{C}y_2)$ on $\R^n$ (by the strong maximum principle) and the function
$$u:=\underline{z}(.,\widetilde{C}y_1)=\overline{z}(.,\widetilde{C}y_2)$$
would be a solution of (\ref{soliton equation C})-(\ref{asymptotic u f}) in Theorem \ref{main thm}.
\subsection{The a priori estimates}
The interior $C^0$ estimate is trivial by (\ref{um between barriers}). We will focus below on the interior gradient and $C^2$ estimates. The higher order estimates will then follow from the Evans-Krylov and Schauder estimates.
\subsubsection{The interior gradient estimate}
Let $K\subset\R^n$ be a fixed compact subset. We aim to obtain a bound
\begin{equation}\label{required int est nu}
\sup_K\nu\leq C_K
\end{equation}
for a controlled constant $C_K<C,$ for all the solutions of the sequence of Dirichlet problems (\ref{seq d problems}) with $m\geq m_K$ sufficiently large. We construct a smooth auxiliary function $\psi:\overline{B}_R\rightarrow\R$ defined on a closed ball $\overline{B}_R$ which contains $K,$  which is strictly convex, spacelike and such that
\begin{equation}\label{def Cpsi}
C_{\psi}:=\frac{1}{\sqrt{1-\sup_{\overline{B}_R}|D\psi|^2}}<C,
\end{equation}
\begin{equation}\label{pty psi est grad int}
\psi\leq \underline{u}-\delta_0\hspace{.3cm}\mbox{in}\hspace{.2cm}K\hspace{.5cm}\mbox{and}\hspace{.5cm}\psi\geq \overline{u}\hspace{.3cm}\mbox{on}\hspace{.2cm}\partial B_R
\end{equation}
for some constant $\delta_0>0.$ For the construction of $\psi,$ let us fix $\delta_0>0$ and $R_0$ such that $K\subset \overline{B}_{R_0},$ and set, for $\widetilde{C}:=\sqrt{1-1/C^2},$
$$\psi(x):=-A_0+\widetilde{C}\sqrt{1+|x|^2}$$
with
$$A_0:=-\inf_{\overline{B}_{R_0}}\underline{u}+\widetilde{C}\sqrt{1+R_0^2}+\delta_0.$$ 
We then have
$$\inf_{\overline{B}_{R_0}}\underline{u}-\sup_{\overline{B}_{R_0}}\psi=\delta_0$$
which in particular implies
$$\underline{u}-\psi\geq\delta_0$$
on $K.$ We moreover have, for all $R>R_0,$
$$\sup_{\overline{B}_{R}}|D\psi|=\widetilde{C}\frac{R}{\sqrt{1+R^2}}<\widetilde{C},$$
and, since
$$\psi(x)=\widetilde{C}|x|+O(1)\hspace{.5cm}\mbox{and}\hspace{.5cm}\overline{u}(x)=\widetilde{C}|x|-\frac{1}{C^2}\sqrt{\frac{n-2}{n}}\log{|x|}+O(1)$$
as $|x|$ tends to infinity, $\psi\geq\overline{u}$ on $\R^n\backslash B_R$ if $R$ is chosen sufficiently large. For such an $R>R_0$ the spacelike function $\psi:\overline{B}_R\rightarrow\R$ is such that (\ref{def Cpsi}) and (\ref{pty psi est grad int}) hold. It is moreover strictly convex.

We consider a solution $u$ of (\ref{seq d problems}) with $m$ large enough so that $\overline{B}_R\subset \Omega_m.$ Let us consider the test function
$$\varphi=\frac{\eta}{C-\nu}$$
with
\begin{equation}\label{expr eta}
\eta=(u-\psi)^A,
\end{equation}
where $A>0$ is a constant to be chosen later.  The function $\varphi$ is non-negative on the compact set $\{u\geq \psi\},$ vanishes on $\{u= \psi\},$ and thus reaches a maximum at a point $x_0$ belonging to $\{u>\psi\}.$ The aim is to prove 
\begin{equation}\label{required int est x0}
\nu(x_0)\leq C'
\end{equation}
for some controlled constant $C'<C.$ Such a bound will indeed imply that $\varphi(x_0)\leq c$ for some other controlled constant $c,$ and thus that $\varphi\leq c$ on $K;$ since $\inf_K\eta\geq \delta_0^A$, this will in turn imply that 
$$\frac{1}{C-\nu}\leq \frac{c}{\delta_0^A}\hspace{.3cm}\mbox{on}\hspace{.3cm}K,$$
and thus will give an estimate of the form (\ref{required int est nu}). We thus focus on the obtention of (\ref{required int est x0}). We keep the notations introduced in Section \ref{subsection max pple gradient}. Since $\varphi$ reaches its maximum at  $x_0,$ we have \begin{equation}\label{maximum conditions grad int}
(\log\varphi)_i=0\hspace{.5cm} \mbox{and}\hspace{.5cm}F^{ij}(\log\varphi)_{ij}\leq 0
\end{equation} 
at that point. Since
$$(\log\varphi)_i=\frac{\eta_i}{\eta}+\frac{\nu_i}{C-\nu}$$
the first condition reads
\begin{equation}\label{extr phi grad int}
\eta_i=-\eta\frac{\nu_i}{C-\nu}.
\end{equation}
Using this equality, we further compute
$$(\log\varphi)_{ij}=\frac{\eta_{ij}}{\eta}-\frac{\eta_i\eta_j}{\eta^2}+\frac{\nu_{ij}}{C-\nu}+\frac{\nu_i\nu_j}{(C-\nu)^2}=\frac{\eta_{ij}}{\eta}+\frac{\nu_{ij}}{C-\nu},$$
and the second condition in (\ref{maximum conditions grad int}) reads
\begin{equation}\label{max phi grad int}
F^{ij}\nu_{ij}+(C-\nu)\frac{1}{\eta}F^{ij}\eta_{ij}\leq 0.
\end{equation}
By (\ref{Fijnuij}) (with $\sigma=1$) and since $F^{ij}h_j^kh_{ki}\geq 0,$ we have
$$F^{ij}\nu_{ij}\geq t^k\nu_k=-\frac{t^k\eta_k}{\eta}(C-\nu).$$
Inequality (\ref{max phi grad int}) thus implies
\begin{equation}\label{max phi grad int p}
F^{ij}\eta_{ij}\leq t^k\eta_k,
\end{equation}
and, since $\nu$ is bounded by $C,$ 
\begin{equation}\label{max phi grad int p 2}
F^{ij}\eta_{ij}\leq cA(u-\psi)^{A-1}.
\end{equation}
Here and below, we use the letter $c$ to denote a generic controlled constant (i.e. a controlled constant which may change during the computations). We now estimate the left-hand side of (\ref{max phi grad int p 2}). Let us take an orthonormal basis $(e_1,e_2,\ldots,e_n)$ formed by principal directions of the graph of $u$ at $x_0.$ Using (\ref{expr eta}), (\ref{extr phi grad int}) and the fact that $\nu_i=\lambda_iu_i$ (by (\ref{nui})) we obtain
\begin{equation}\label{extr phi explicit}
A\frac{u_i-\psi_i}{u-\psi}=-\frac{\lambda_iu_i}{C-\nu}.
\end{equation}
Let us consider
$$J=\{i:\ (1-\alpha)u_i^2>\psi_iu_i\}\hspace{.5cm}\mbox{and}\hspace{.5cm}J'=\{i:\ (1-\alpha)u_i^2\leq \psi_iu_i\}.$$
where $\alpha>0$ is a small constant such that
\begin{equation}\label{cond alpha}
1-\alpha>\sup_{\overline{B}_R}|D\psi|
\end{equation}
that will be chosen later. We may assume that $J$ is not empty, since on the contrary we would have
$$(1-\alpha)|u_i|\leq |\psi_i|,\hspace{.5cm} \forall i=1,\ldots,n,$$
which would imply (recall (\ref{nu grad u}))
$$\nu^2-1=|\nabla u|^2\leq\frac{1}{(1-\alpha)^2}|\nabla\psi|^2\leq \frac{1}{(1-\alpha)^2}\nu^2|D\psi|^2$$
and thus
$$\nu^2\leq\frac{1}{1-\frac{1}{(1-\alpha)^2}|D\psi|^2}\leq \frac{1}{1-\frac{1}{(1-\alpha)^2}\sup_{\overline{B}_R}|D\psi|^2},$$
i.e., recalling (\ref{def Cpsi}), the required estimate (\ref{required int est x0}) if $\alpha$ is chosen sufficiently small.
\\Using (\ref{extr phi explicit}), we obtain that, for all $i\in J,$
\begin{equation}\label{lambdai negative}
\lambda_i=-A\frac{(C-\nu)}{u-\psi}\left(1-\frac{\psi_i}{u_i}\right)<-A\frac{(C-\nu)}{u-\psi}\alpha
\end{equation}
is negative. We compute
\begin{eqnarray}
F^{ij}\eta_{ij}&=&A(u-\psi)^{A-1}\left(F^{ij}u_{ij}-F^{ij}\psi_{ij}\right)\label{Fijetaij}\\
&&+A(A-1)(u-\psi)^{A-2}F^{ij}(u_i-\psi_i)(u_j-\psi_j).\nonumber
\end{eqnarray}
Since $u_{ij}=\nu h_{ij},$ we have
\begin{equation}\label{Fijuij}
F^{ij}u_{ij}=\nu F\geq 0.
\end{equation}
To estimate the term $F^{ij}\psi_{ij}$ we will use the formula
$$\sum_i\sigma_{1,i}\psi_{ii}=\sigma_2\langle D\psi,Du\rangle \nu+\sum_i\sigma_{1,i}\sum_{\alpha,\beta}\frac{\partial^2\psi}{\partial x_{\alpha}\partial x_{\beta}}\langle e^0_{\alpha},e_i\rangle\langle e^0_{\beta},e_i\rangle$$
where $\sigma_{1,i}=\sum_{k\neq i}\lambda_k$ and $e^0_{\alpha},$ $\alpha=1,\ldots,n$ denotes the canonical basis of $\R^n.$ A proof may be found in \cite{U}. Since $\nu$ is bounded by $C$ and since $|\langle e^0_{\alpha},e_i\rangle|\leq \nu$ and the derivatives of $\psi$ are under control, it implies that
\begin{equation}\label{est psiii inter}
\sum_i\sigma_{1,i}\psi_{ii}\leq c(\sigma_2+\sigma_1).
\end{equation}
Now we may assume that $\sigma_2\leq 1$ (since on the contrary the estimate (\ref{required int est x0}) would be a direct consequence of the partial differential equation (\ref{seq d problems})), which implies by the Mac-Laurin inequality (\ref{mac laurin}) that
$$\sigma_2\leq\sqrt{\sigma_2}\leq \sqrt{\frac{n-1}{2n}}\sigma_1.$$
This inequality together with (\ref{est psiii inter}) implies the estimate
\begin{equation}\label{Fijpsiij}
F^{ij}\psi_{ij}\leq c\sigma_2^{-\frac{1}{2}}\sigma_1
\end{equation}
for a controlled constant $c.$ So (\ref{Fijetaij}), (\ref{Fijuij}) and (\ref{Fijpsiij}) imply that
\begin{equation}\label{est Fijetaij}
F^{ij}\eta_{ij}\geq -A(u-\psi)^{A-1}c\sigma_2^{-\frac{1}{2}}\sigma_1+\frac{A(A-1)}{2}(u-\psi)^{A-2}\sigma_2^{-\frac{1}{2}}\sum_i\sigma_{1,i}(u_i-\psi_i)^2.
\end{equation}
Inequalities (\ref{max phi grad int p 2}) and (\ref{est Fijetaij}) yield
$$(A-1)\sum_i\sigma_{1,i}(u_i-\psi_i)^2\leq c\left(\sigma_1+\sigma_2^{\frac{1}{2}}\right)(u-\psi)$$
and thus, since $\sigma_2^{\frac{1}{2}}\leq c\sigma_1$ and $u-\psi$ is bounded,
$$(A-1)\sum_i\sigma_{1,i}(u_i-\psi_i)^2\leq c\sigma_1$$
for a further controlled constant $c.$ For $i\in J$ we have $(u_i-\psi_i)^2\geq u_i^2\alpha^2$ and $\sigma_{1,i}\geq\sigma_1$ (since then $\lambda_i\leq 0$ by (\ref{lambdai negative})), and thus
\begin{equation}\label{ineq i in J}
(A-1)\alpha^2\sum_{i\in J}u_i^2\leq c.
\end{equation}
Note also that $(1-\alpha)|u_i|\leq|\psi_i|$ for all $i\in J',$ which gives
\begin{equation}\label{ineq i in Jp}
\sum_{i\in J'}u_i^2\leq\frac{1}{(1-\alpha)^2}\sum_{i\in J'}\psi_i^2\leq\frac{1}{(1-\alpha)^2}|\nabla\psi|^2\leq\nu^2\frac{1}{(1-\alpha)^2}|D\psi|^2.
\end{equation}
From (\ref{ineq i in J}) and (\ref{ineq i in Jp}) we deduce
$$|\nabla u|^2\leq \frac{c}{(A-1)\alpha^2}+\nu^2\frac{1}{(1-\alpha)^2}|D\psi|^2,$$
and since $|\nabla u|^2 =\nu^2-1$ that
\begin{equation}\label{estim nu int interm} 
\nu^2\leq\left(1+ \frac{c}{(A-1)\alpha^2}\right)\ \frac{1}{1-\frac{1}{(1-\alpha)^2}|D\psi|^2}.
\end{equation}
Let $C',C''$ be constants such that
$$C_{\psi}:=\frac{1}{\sqrt{1-\sup_{\overline{B}_R}|D\psi|^2}}<C''<C'<C.$$
Choosing $\alpha=\alpha(C'')$ sufficiently small so that
$$\frac{1}{\sqrt{1-\frac{1}{(1-\alpha)^2}\sup_{\overline{B}_R}|D\psi|^2}}<C'',$$
we deduce from (\ref{estim nu int interm}) that
$$\nu^2\leq \left(1+ \frac{c}{(A-1)\alpha^2}\right)C''^2.$$
We afterwards choose $A=A(\alpha,c,C',C'')$ sufficiently large so that 
$$\left(1+ \frac{c}{(A-1)\alpha^2}\right)C''^2\leq C'^2.$$
This implies $\nu\leq C',$ which is the required estimate (\ref{required int est x0}), and the interior $C^1$ estimate (\ref{required int est nu}) follows.

\subsubsection{The interior $C^2$ estimate}
Here again, the estimate relies on the $C^2$ estimate of Urbas \cite{U}. We need first to construct an auxiliary smooth spacelike function $\psi_2:\overline{B}_R\rightarrow\R$ defined on a large ball $\overline{B}_R$ containing $K,$ strictly convex and such that 
\begin{equation}\label{cond psi int C2 1}
\psi_2\geq \overline{u}+1\hspace{.5cm} \mbox{on}\hspace{.3cm} K
\end{equation} 
and 
\begin{equation}
\label{cond psi int C2 2}\psi_2<\underline{u}\hspace{.5cm}\mbox{on}\hspace{.3cm}\partial B_R.
\end{equation}
Let us fix $R_0$ such that $K\subset B_{R_0}$ and a constant $C'$ smaller than $C,$ and consider, for $\widetilde{C'}=\sqrt{1-1/{C'^2}},$ the function
$$\psi_2(x)=A+\sqrt{1+\widetilde{C'}^2|x|^2}$$
on $\R^n,$ where $A$ is a large constant such that (\ref{cond psi int C2 1}) holds. Since
$$\psi_2(x)\sim\widetilde{C'}|x|\hspace{.5cm}\mbox{and}\hspace{.5cm}\underline{u}(x)\sim\widetilde{C}|x|$$
as $|x|$ tends to infinity, with $\widetilde{C'}<\widetilde{C},$ (\ref{cond psi int C2 2}) also holds if $R$ is sufficiently large. We then follow the lines of the $C^2$ estimate in \cite{U}: we set 
$$\eta:=\psi_2-u,$$
and apply the maximum principle to the function $\widetilde{W}$ introduced in (\ref{def W tilde}), with that new definition of $\eta,$ on the set $\{\psi_2\geq u\}$ (which is compact by (\ref{cond psi int C2 2})). The estimate in Lemma \ref{lem est W tilde} permits to carry out the method in \cite{U}, without any modification: $\eta^\beta\lambda_1$ is controlled at a point where $\widetilde{W}$ reaches its maximum, which implies that $\widetilde{W}$ is bounded by a controlled constant; see Section \ref{section max pp C2} and especially Remark \ref{rem int C2} in that section. Using (\ref{cond psi int C2 1}) we deduce  an upper bound of $D^2u$ on $K$ by a controlled constant.
\appendix
\section{Auxiliary results on classical ODE's}
We gather here auxiliary results concerning the asymptotic behavior of solutions of some classical ODE's. Although these results are elementary and probably well known, and since we don't find them in the literature, we include the proofs.
\begin{prop}\label{prop 1 app 1}
1. Assume that $z:(r_0,+\infty)\rightarrow\R$ is a positive solution of 
\begin{equation}\label{eqn z z2 A0 B0}
z'=A_0z^2+B_0,
\end{equation}
with $A_0,B_0$ constants such that $A_0<0$ and $B_0>0.$ Then
$$\lim_{r\rightarrow+\infty}z(r)=\sqrt{-\frac{B_0}{A_0}}.$$
2. For all $r_0\geq 0$ and $z_0\geq 0,$ Equation (\ref{eqn z z2 A0 B0}) admits a unique solution $z:[r_0,+\infty)\rightarrow\R$ such that $z(r_0)=z_0;$ it is positive on $(r_0,+\infty).$
\end{prop}
\begin{proof}
1. Assuming that $z\neq\sqrt{-\frac{B_0}{A_0}}$ (on the contrary the result is obvious), we have
$$dr=\frac{dz}{A_0z^2+B_0}=\frac{dz}{2\sqrt{-\frac{B_0}{A_0}}A_0}\left\{\frac{1}{z-\sqrt{-\frac{B_0}{A_0}}}-\frac{1}{z+\sqrt{-\frac{B_0}{A_0}}}\right\}$$
and thus
\begin{equation}\label{eqn prop z implicite}
\left|\frac{z-\sqrt{-\frac{B_0}{A_0}}}{z+\sqrt{-\frac{B_0}{A_0}}}\right|=C\exp\left(2\sqrt{-\frac{B_0}{A_0}}A_0(r-r_0)\right)
\end{equation}
for some constant $C>0.$ Letting $r$ tend to $+\infty,$ the right-hand term tends to 0, and thus
$$\frac{z-\sqrt{-\frac{B_0}{A_0}}}{z+\sqrt{-\frac{B_0}{A_0}}}\rightarrow_{r\rightarrow+\infty} 0,$$
which gives the result.
\\2. Equation (\ref{eqn prop z implicite}) yields
\begin{equation}\label{eqn prop z explicite}
z=-\sqrt{-\frac{B_0}{A_0}}+\frac{2\sqrt{-\frac{B_0}{A_0}}}{1-C\exp\left(2\sqrt{-\frac{B_0}{A_0}}A_0(r-r_0)\right)}
\end{equation}
for some constant $C\in\R.$ Conversely, (\ref{eqn prop z explicite}) with
$$C=1-\frac{2\sqrt{-\frac{B_0}{A_0}}}{z_0+\sqrt{-\frac{B_0}{A_0}}}$$
defines a solution $z:[r_0,+\infty)\rightarrow\R$ of (\ref{eqn z z2 A0 B0}) such that $z(r_0)=z_0.$ Since $z$ is increasing if $z_0\in[0,\sqrt{-\frac{B_0}{A_0}})$ (this indeed corresponds to $C\in[-1,0)$), decreasing if $z_0>\sqrt{-\frac{B_0}{A_0}}$ (this corresponds to $C\in (0,1)$) and constant if $z_0=\sqrt{-\frac{B_0}{A_0}}$ (i.e. for $C=0$), it follows from the positive sign of the initial condition $z_0$ or of the limit at infinity that it is positive on $(r_0,+\infty).$
\end{proof}
\begin{prop} \label{prop 2 app 1}
Assume that $z:(0,+\infty)\rightarrow\R$ is a positive solution of 
$$z'=Az^2+B$$
where $A,B:(0,+\infty)\rightarrow\R$ are continuous functions such that 
$$A(r)\rightarrow_{r\rightarrow+\infty} A_0<0\hspace{.5cm}\mbox{and}\hspace{.5cm}B(r)\rightarrow_{r\rightarrow+\infty} B_0>0.$$
Then 
$$\lim_{r\rightarrow+\infty}z(r)=\sqrt{-\frac{B_0}{A_0}}.$$
\end{prop}
\begin{proof}
Let us fix constants $A_1,$ $A_2,$ $B_1,$ $B_2$ such that
$$A_1<A_0<A_2<0\hspace{.5cm}\mbox{and}\hspace{.5cm}0<B_1<B_0<B_2.$$
For $r\geq r_0$ sufficiently large,
$$A_1z^2+B_1<z'<A_2z^2+B_2.$$
Let us consider $z_1$ and $z_2$ solutions of
$$z_1'=A_1z_1^2+B_1\hspace{.5cm}\mbox{and}\hspace{.5cm}z_2'=A_2z_2^2+B_2$$
such that
$$z_1(r_0)=z_2(r_0)=z(r_0)$$
(by Proposition \ref{prop 1 app 1}, 2.). By comparison we have 
$$z_1<z<z_2\hspace{.5cm}\mbox{in}\hspace{.5cm}(r_0,+\infty)$$
and thus, by Proposition \ref{prop 1 app 1}, 1.,
$$\liminf z\geq\liminf z_1=\sqrt{-\frac{B_1}{A_1}}$$
and
$$\limsup z\leq\limsup z_2=\sqrt{-\frac{B_2}{A_2}}.$$
We finally take the limits $A_1,A_2\rightarrow A_0$ and $B_1,B_2\rightarrow B_0$ to obtain the result.
\end{proof}
\begin{prop}\label{prop 3 app 1}
If $y:[r_0,+\infty)\rightarrow\R$ is a solution of
$$y'+ay=b$$
where $a,b:[r_0,+\infty)\rightarrow\R$ are continuous functions such that   
$$0<a_0\leq a\hspace{.5cm}\mbox{and}\hspace{.5cm}|b|\leq\|b\|_{\infty}<+\infty$$
for some constant $a_0,$ then $y$ is bounded on $[r_0,+\infty):$
$$\|y\|_{\infty}\leq C(y(r_0),a_0,\|b\|_{\infty}).$$
\end{prop}
\begin{proof}
Setting, for $r\geq r_0,$
$$A(r)=\int_{r_0}^ra(t)dt$$
we have
$$y(r)=\left(y(r_0)+\int_{r_0}^rb(u)e^{A(u)}du\right)e^{-A(r)}.$$
We have $A(r)\geq 0$ and
$$A(u)-A(r)\leq-a_0(r-u)$$
if $r_0\leq u\leq r,$ which implies that
$$|y(r)|\leq |y(r_0)|+\|b\|_{\infty}\int_{r_0}^re^{-a_0(r-u)}du\leq  |y(r_0)|+\frac{\|b\|_{\infty}}{a_0}.$$
\end{proof}
\noindent\textbf{Acknowledgments:}  The author thanks the referee for a simplification in the proof of Lemma \ref{lem v eps}. He also thanks the project PAPIIT-UNAM IA106218 for support.

\end{document}